\documentclass[11pt]{article}

\usepackage[margin=1.1in]{geometry}
\usepackage{amsmath,amssymb,amsthm}
\usepackage{array,booktabs,tabularx}
\usepackage[colorlinks=true,linkcolor=blue,citecolor=blue,urlcolor=blue]{hyperref}

\newtheorem{theorem}{Theorem}
\newtheorem{lemma}[theorem]{Lemma}
\newtheorem{corollary}[theorem]{Corollary}
\newtheorem{proposition}[theorem]{Proposition}
\theoremstyle{remark}
\newtheorem{remark}[theorem]{Remark}

\newcommand{\defi}{\operatorname{def}}
\DeclareMathOperator{\exc}{exc}
\newcommand{\floor}[1]{\left\lfloor #1\right\rfloor}
\newcommand{\ceil}[1]{\left\lceil #1\right\rceil}
\newcommand{\clawconstant}{c_{\mathrm{claw}}}

\title{Improved Bounds for Unavoidable Claws in Tournaments}
\author{
Jiangdong Ai\thanks{School of Mathematical Sciences and LPMC, Nankai University,
Tianjin 300071, P.R. China. Email: \href{mailto:jd@nankai.edu.cn}{\nolinkurl{jd@nankai.edu.cn}}.
The research of this author was supported by the National Natural Science Foundation
of China No. 12522117.}
\and
Yongxin Lan\thanks{School of Science, Hebei University of Technology,
Tianjin 300401, P.R. China. Email: \href{mailto:yxlan@hebut.edu.cn}{\nolinkurl{yxlan@hebut.edu.cn}}.}
}
\date{}

\begin{document}
\maketitle

\begin{abstract}
Let $u(n)$ be the largest integer $d$ such that every $n$-vertex claw with at most $d$ branches occurs in every tournament on $n$ vertices, and let $\clawconstant=\limsup_{n\to\infty}u(n)/n$.
In 1998, Lu, Wang and Wong proved that $19/50\le\clawconstant\le11/23$, and these have remained the best bounds known. We improve them to $2/5\le\clawconstant\le10/21$. We also isolate two parameters $\theta$ and $\sigma$ which place the lower- and upper-bound arguments in a common framework: we show $\sigma\le\theta$ and $1/21\le\sigma\le\theta\le1/5$, our two bounds being the images of the endpoints under $\alpha\mapsto\frac12-\frac{\alpha}{2}$, and $\sigma=\theta$ would force $\lim u(n)/n$ to exist.
\end{abstract}

\section{Introduction}\label{sec:intro}

An oriented graph $H$ on $n$ vertices is \emph{$n$-unavoidable} if every tournament on $n$ vertices contains a copy of $H$; otherwise it is \emph{$n$-avoidable}. R\'edei's theorem that every tournament contains a directed Hamilton path is the classical first example \cite{Redei}. Saks and S\'os initiated a systematic study of spanning unavoidable digraphs and, in particular, rooted directed trees whose branches are directed paths \cite{SaksSos}.

Let $\lambda=(\lambda_1,\ldots,\lambda_k)$ be a partition of $n-1$ into positive integers. The \emph{claw} $C(\lambda)$ is obtained from $k$ pairwise vertex-disjoint directed paths of lengths $\lambda_1,\ldots,\lambda_k$ by identifying their initial vertices. The common vertex is the root, the paths are the branches, and the degree of the claw is $k$. For $k\ge2$, this is also the maximum degree of the underlying tree; the distinction for a one-branch claw will not matter below. We abbreviate repeated parts by exponents: $C(3^{r},2^{s})$ denotes the claw with $r$ branches of length $3$ and $s$ branches of length $2$; it has degree $r+s$ and $3r+2s+1$ vertices. An exponent equal to $0$ means that the corresponding part is absent. The \emph{excess} of $\lambda$ is $\exc(\lambda)=\sum_i(\lambda_i-2)^+$, which measures by how much the branches of $C(\lambda)$ exceed length $2$.

Saks and S\'os conjectured that every $n$-vertex claw of degree at most $\floor{n/2}$ is $n$-unavoidable. Lu disproved this conjecture and proved successively that degree at most $n/4$ and then $3n/8$ is sufficient \cite{Lu1991,Lu1993}; he also obtained the asymptotic upper bound $25/52$ \cite{Lu1996}. Lu, Wang and Wong later proved that every claw of degree less than $19n/50$ is unavoidable and constructed avoidable claws of degree asymptotic to $11n/23$ \cite{LuWangWong}. These are still the bounds recorded in later work and in the 2024 survey of Stein \cite{MycroftNaia,SteinSurvey}. As far as we are aware, the present results are the first improvements to either side since 1998.

Let $u(n)$ be the largest integer $d$ such that every $n$-vertex claw of degree at most $d$ is $n$-unavoidable, and let $\clawconstant=\limsup_{n\to\infty}u(n)/n$. Thus the previous bounds read $19/50\le\clawconstant\le11/23$. We improve both sides.

\begin{theorem}\label{thm:lower}
For every $n\ge5$, every $n$-vertex claw of degree at most $\floor{2n/5}$ is $n$-unavoidable.
\end{theorem}

\begin{theorem}\label{thm:upper}
Let $n\ge441$, and write $n=21m+\ell$, where $0\le\ell\le20$. There is an $n$-vertex tournament which does not contain $C\bigl(3^{m-\ell-1},2^{9m+2\ell+1}\bigr)$.
This claw has degree $10m+\ell$, and consequently $u(n)\le10m+\ell-1$.
\end{theorem}

Theorems~\ref{thm:lower} and \ref{thm:upper} give $\liminf_{n\to\infty}u(n)/n\ge2/5$ and $\limsup_{n\to\infty}u(n)/n\le10/21$, respectively. In particular, we obtain the following interval.

\begin{corollary}\label{cor:interval}
We have $\frac25\le\clawconstant\le\frac{10}{21}$.
\end{corollary}

The main ingredient for the lower bound is the following spanning structure.

\begin{theorem}\label{thm:structural}
Let $T$ be a tournament on $n\ge2$ vertices and let $N=n-1$. Then $T$ contains a spanning claw $C(\mu)$ such that every part of $\mu$ belongs to $\{1,2,3\}$, at most $\floor{N/5}$ parts are equal to $3$, and $\mu$ has at least $\floor{2(N+1)/5}=\floor{2n/5}$ parts.
\end{theorem}

For a maximum-outdegree vertex $v$, the arcs from $N^+(v)$ to $N^-(v)$ form a bipartite graph. A capacity-two Hall assignment gives a spanning claw rooted at $v$ of depth at most $3$, in which a vertex of $N^+(v)$ carrying $0$, $1$ or $2$ vertices of $N^-(v)$ begins a branch of length $1$, $2$ or $3$; the number of length-$3$ branches is exactly the Hall deficiency $d$ at $v$. Since a branch of length $3$ spends three vertices to produce a single branch, small deficiency means many branches. If $d$ exceeds $N/5$, then a maximum-deficiency witness $Y\subseteq N^-(v)$ forces the arcs from $A=\Gamma(Y)$ to $Y$ to be dense, which in turn produces a vertex $w\in A$ of large outdegree in $T[A\cup Y]$. Re-rooting the claw at $w$ and absorbing the vertices outside $A\cup Y$ below $N^+(w)\cap Y$ then yields a spanning claw with at most $\floor{N/5}$ branches of length $3$. Theorem~\ref{thm:lower} follows from Theorem~\ref{thm:structural} and the dominance order on partitions: a target claw $C(\lambda)$ of degree at most $\floor{2n/5}$ dominates the balanced claw with the same number of branches, which in turn dominates the claw supplied by Theorem~\ref{thm:structural}.

The upper bound has a different form. Lemma~\ref{lem:rootobstruction} bounds the Hall deficiency at the root of a spanning claw $C(\lambda)$ from above by $\exc(\lambda)$, while Lemma~\ref{lem:weightedtemplate} bounds the Hall deficiency from below at \emph{every} vertex of a weighted blow-up, using one Hall witness per class of the template. Choosing $\lambda$ so that the two bounds conflict rules out $C(\lambda)$ altogether. Section~\ref{sec:bottleneck} makes the symmetry between the two arguments precise: they are governed by parameters $\theta$ and $\sigma$ with $\sigma\le\theta$, and our two theorems say $\theta\le1/5$ and $\sigma\ge1/21$.

\section{Preliminaries}\label{sec:prelim}

All tournaments and digraphs are finite. For a digraph $D$ and $v\in V(D)$, we write $N_D^+(v)$ and $N_D^-(v)$ for the out- and in-neighbourhood of $v$ and $d_D^+(v)=|N_D^+(v)|$; the subscript is
omitted when $D$ is clear from the context. For disjoint vertex sets $X,Y$ in a tournament, write $X\to Y$ if every arc between them is directed from $X$ to $Y$, and let $e(X,Y)$ be the number of arcs from $X$ to $Y$.

A partition is a nonincreasing sequence of positive integers; trailing zeros are appended when two partitions of the same integer are compared. If $\lambda$ and $\mu$ have the same sum, then $\lambda$ \emph{dominates} $\mu$, written $\lambda\succeq\mu$, if $\sum_{i=1}^j\lambda_i\ge\sum_{i=1}^j\mu_i$ for every $j\ge1$. We use the following result of Saks and S\'os \cite[Lemma 2.1]{SaksSos}.

\begin{lemma}\label{lem:dominance}
Let $\lambda$ and $\mu$ be partitions of the same integer with $\lambda\succeq\mu$. If a tournament $T$ contains $C(\mu)$, then it also contains $C(\lambda)$.
\end{lemma}

For positive integers $N$ and $k\le N$, the balanced $k$-part partition of $N$ is the unique partition whose largest and smallest parts differ by at most one.

\begin{lemma}\label{lem:balanced}
Every $k$-part partition of $N$ dominates the balanced $k$-part partition of $N$.
\end{lemma}

\begin{proof}
If a $k$-part partition is not balanced, transfer one unit from a largest part to a smallest part and reorder. This weakly decreases the partition in the dominance order, and iteration ends at the balanced partition.
\end{proof}

For a bipartite graph $G=(P,Q;E)$ and $Y\subseteq Q$, let
$\Gamma_G(Y)=\{p\in P:\ py\in E\text{ for some }y\in Y\}$ be the neighbourhood
of $Y$; we drop the subscript when $G$ is clear from the context, and write
$\Gamma_{B}(Y)=\Gamma_G(Y)\cap B$ for $B\subseteq P$. Let $\nu(G)$ be the
maximum matching number of $G$ and define
\[
\defi(G)=|Q|-\nu(G)=\max_{Y\subseteq Q}\bigl(|Y|-|\Gamma(Y)|\bigr).
\]
The second equality is the deficiency form of Hall's theorem, due to Ore~\cite{Ore1955}.

For a map $f\colon Q\to P$ and $p\in P$, the \emph{load} of $p$ is $|f^{-1}(p)|$, and $p$ is \emph{used} if its load is positive. We say $f$
has \emph{capacity} $2$ if every load is at most $2$.

\begin{lemma}\label{lem:capacity}
Let $G=(P,Q;E)$ satisfy $|Y|\le2|\Gamma(Y)|$ for every $Y\subseteq Q$. Then
there is a map $f\colon Q\to P$ of capacity $2$ with $qf(q)\in E$ for every
$q\in Q$, and exactly $\defi(G)$ vertices of $P$ have load $2$.
\end{lemma}

\begin{proof}
Replace each vertex of $P$ by two copies. The hypothesis and Hall's theorem give a map $f:Q\to P$ of load at most $2$. Choose $f$ so that
its support $S=\{p\in P:f^{-1}(p)\ne\varnothing\}$ has maximum size, and select one representative edge at each vertex of $S$. These edges form a matching $M$ of size $|S|$, so $|S|\le\nu(G)$.

Suppose that $|S|<\nu(G)$, and let
$q_0p_1q_1\cdots q_{r-1}p_r$ be an $M$-augmenting path, where
$p_iq_i\in M$ for $1\le i<r$. Since $q_0$ is uncovered by $M$ but
assigned by $f$, it is the nonrepresentative member of a load-$2$ vertex. Simultaneously reassign $q_{i-1}$ to $p_i$ for
$1\le i\le r$. Every previously used vertex remains used: each intermediate vertex exchanges its representative for the preceding vertex on the path, while the old parent of $q_0$ either retains its
other child or is one of the intermediate vertices. On the other hand, $p_r\notin S$ becomes used, contradicting the maximality of
$|S|$. Hence $|S|=\nu(G)$.

Since every positive load is $1$ or $2$, the number of vertices with load $2$ is
\[
|Q|-|S|=|Q|-\nu(G)=\defi(G).
\]
\end{proof}

\section{A shallow spanning claw}\label{sec:shallow}

\begin{lemma}\label{lem:maxroot}
Let $S$ be a tournament on at least two vertices and let $v$ have maximum outdegree. Let $P=N_S^+(v)$ and $Q=N_S^-(v)$, and let $G_v=(P,Q;E)$ be the bipartite graph in which $uz\in E$ if and only if $u\to z$. If $d=\defi(G_v)$, then $S$ contains a spanning claw rooted at $v$ whose branches have lengths in $\{1,2,3\}$, exactly $d$ of them having length $3$.
\end{lemma}

\begin{proof}
Let $Z\subseteq Q$ be nonempty, let $A=\Gamma(Z)$ and $B=P\setminus A$, and write $\zeta=|Z|$, $a=|A|$, $b=|B|$. Since $Z\to B$ and $Z\to v$, while every vertex of $Z$ has outdegree at most $|P|=a+b$, summing outdegrees over $Z$ gives $\zeta(a+b)\ge \zeta+\zeta b+\binom{\zeta}{2}$, hence $\zeta a\ge\frac{\zeta(\zeta+1)}2$.
Thus $|Z|\le2|\Gamma(Z)|-1$, so Lemma~\ref{lem:capacity} applies. Assign every vertex of $Q$ to an inneighbour in $P$, with every parent of load at most $2$ and exactly $d$ parents of load $2$. Parents of load $0$, $1$ or $2$ give branches of lengths $1$, $2$ or $3$, respectively; in the last case the two successors are ordered according to the arc between them. These branches are internally disjoint and cover $V(S)$.
\end{proof}

\begin{proof}[Proof of Theorem~\ref{thm:structural}]
Let $T$ have $N+1$ vertices and let $r_0=\floor{N/5}$. Choose a maximum-outdegree vertex $x$, write $P=N_T^+(x)$, $Q=N_T^-(x)$, $p=|P|$, $q=|Q|$, and let $d=\defi(G_x)$ for the bipartite graph from $P$ to $Q$. If $d\le r_0$, Lemma~\ref{lem:maxroot} gives a spanning claw with at most $r_0$ branches of length $3$ and $p\ge\ceil{N/2}\ge\floor{2(N+1)/5}$ branches in total. Assume therefore that $d>r_0$, so
\begin{equation}\label{eq:N5d}
N<5d.
\end{equation}
Choose $Y\subseteq Q$ with $|Y|-|\Gamma(Y)|=d$. Let $A=\Gamma(Y)$, $B=P\setminus A$, $C=Q\setminus Y$, and write $a=|A|$, $b=|B|$, $c=|C|$, $y=|Y|=a+d$ and $g=e(A,Y)$. Since $Y\to B$ and $Y\to x$, summing the outdegrees of the vertices of $Y$ gives $y(a+b)\ge y+yb+\binom y2+ay-g$, and hence $g\ge y(y+1)/2$. As $g\le ay$, we have $a\ge d+1$. Since every vertex of $A$ also has outdegree at most $a+b$, summing their outdegrees and using $a(a+b)\ge\binom a2+g$, we obtain
\begin{equation}\label{eq:basic}
b\ge d+\frac{d(d+1)}{2a}.
\end{equation}
Let $h=d(d+1)/(2a)$, $\eta=b-d-h\ge0$ and $\xi=\eta+c$. These quantities are nonnegative real numbers and need not be integral. We have
\begin{equation}\label{eq:Ndecomp}
N=2a+2d+h+\xi.
\end{equation}

There is a matching from $B$ to $C$ saturating $C$, where $uv$ is an edge, for $u\in B$ and $v\in C$, when $u\to v$. Otherwise Hall's theorem gives $C_0\subseteq C$ with $|\Gamma_B(C_0)|<|C_0|$, while the neighbourhood of $Y\cup C_0$ in $P$ is contained in $A\cup\Gamma_B(C_0)$; this yields deficiency greater than $d$, a contradiction.

Consider $U=T[A\cup Y]$. Every $u\in Y$ points to $x$ and to $B$, so $d_U^+(u)\le a-1$. On the other hand, the average outdegree in $U$ of a vertex of $A$ is at least
\[
\frac{\binom a2+g}{a}\ge \frac{a-1}{2}+\frac{(a+d)(a+d+1)}{2a}=a+d+h>a-1.
\]
Thus a maximum-outdegree vertex $w$ of $U$ belongs to $A$. Write $P'=N_U^+(w)$, $Q'=N_U^-(w)$, $p'=|P'|$, $q'=|Q'|$, and let $S=P'\cap Y$, $s=|S|$. Since $w$ has at most $a-1$ outneighbours in $A$ and $|U|=2a+d$, the preceding inequality gives
\begin{equation}\label{eq:secondroot}
p'\ge a+d+h,
\qquad s\ge d+1+h,
\qquad q'\le a-1-h.
\end{equation}
Moreover, (\ref{eq:N5d}) and (\ref{eq:Ndecomp}) give $\xi<3d-2a-h$. Since $a+3h\ge2\sqrt{3d(d+1)/2}>2d$, it follows that
\begin{equation}\label{eq:eslack}
\xi<d-a+2h.
\end{equation}

Apply Lemma~\ref{lem:maxroot} to $U$ with root $w$. Let $\delta'$ be the number of length-$3$ branches and $\nu'=q'-\delta'$ the number of used parents in $P'$. Let $Z_0\subseteq S$ be the set of zero-load vertices. Since at most $\nu'$ vertices of $S$ are used as parents, $|Z_0|\ge s-\nu'=s-q'+\delta'$.
By (\ref{eq:secondroot}) and (\ref{eq:eslack}), $s-q'\ge d-a+2+2h>\xi+2\ge c$. Choose distinct vertices $y_1,\ldots,y_c\in Z_0$, and let $u_i\to v_i$ for $1\le i\le c$ be a matching from $B$ to $C$ saturating $C$. Since $S\subseteq Y$ and $Y\to B$, extend the corresponding branches to $w\to y_i\to u_i\to v_i$ where $1\le i\le c$.
This covers $C$ and $c$ vertices of $B$.

It remains to insert the $b-c+1$ vertices of $\{x\}\cup(B\setminus\{u_1,\ldots,u_c\})$. Every vertex of $S$ points to all of them. Before the preceding extensions the unused capacity in $S$ is at least $2s-q'$, every vertex of $S$ having capacity $2$, and those extensions use $2c$ units. Thus the remaining capacity exceeds the number $b-c+1$ of vertices still to be inserted, since $b=d+h+\eta$ and (\ref{eq:secondroot}), (\ref{eq:eslack}) give
\[
(2s-q'-2c)-(b-c+1)=2s-q'-(b+c+1)\ge d-a+2+2h-\xi>2.
\]
Hence all remaining vertices can be assigned to $S$, with every parent of load at most $2$; two successors are ordered according to the arc between them. We obtain a spanning claw of $T$ whose branches have length at most $3$.

Let $t$ be the number of length-$3$ branches. The claw in $U$ contributes $\delta'$, the branches through $C$ contribute $c$, and after these extensions at least $|Z_0|-c$ vertices of $S$ still have load zero. Using these parents first for the remaining external vertices gives
\begin{align*}
t&\le\delta'+c+\max\{0,b-c+1-(|Z_0|-c)\}\\
 &\le c+\max\{\delta',b+1-s+q'\}.
\end{align*}
Now $\delta'\le q'\le a-1-h$, while $b+1-s+q'\le a+\eta-1-h$. Consequently,
\[
t\le a+\xi-1-h=N-a-2d-2h-1.
\]
Since $a+2h=a+d(d+1)/a\ge2\sqrt{d(d+1)}>2d$, we have $t<N-4d-1<N/5$ by (\ref{eq:N5d}); hence $t\le\floor{N/5}$.

If the resulting claw has $k$ branches, then $N\le3t+2(k-t)=2k+t$, and therefore
\[
k\ge\ceil{\frac{N-\floor{N/5}}2}\ge\floor{\frac{2(N+1)}5},
\]
where the last inequality follows by considering $N$ modulo $5$. This proves the theorem.
\end{proof}

\section{The lower bound}\label{sec:lower}

\begin{proof}[Proof of Theorem~\ref{thm:lower}]
Let $T$ be a tournament on $n\ge5$ vertices, let $N=n-1$, and let $C(\lambda)$ be an $n$-vertex claw with $k\le\floor{2(N+1)/5}$ branches. Let $\beta$ be the balanced $k$-part partition of $N$. Lemma~\ref{lem:balanced} gives $\lambda\succeq\beta$.

By Theorem~\ref{thm:structural}, $T$ contains a spanning claw $C(\mu)$ such that every part of $\mu$ is at most $3$, at most $\floor{N/5}$ parts are equal to $3$, and $\mu$ has at least $\floor{2(N+1)/5}\ge k$ parts. We claim that $\beta\succeq\mu$. If $k\le N/3$, every part of $\beta$ is at least $3$, so the claim is immediate. Suppose that $k>N/3$. Since $k\le\floor{2(N+1)/5}$ and $n\ge5$, we also have $k\le N/2$, $\beta=(3^r,2^{k-r})$, and $r=N-2k$.
The inequality $5k\le2N+2$ gives $5r\ge N-4$, and hence $r\ge\ceil{(N-4)/5}=\floor{N/5}$. Let $t$ be the number of parts of $\mu$ equal to $3$, so $t\le r$. For $j\le r$, we have $\sum_{i=1}^j\beta_i=3j\ge\sum_{i=1}^j\mu_i$ since every part of $\mu$ is at most $3$, while for $r<j\le k$ we have $\sum_{i=1}^j\mu_i\le3t+2(j-t)\le2j+r=\sum_{i=1}^j\beta_i$; for $j>k$, the prefix sum of $\beta$ is already $N$. Thus $\beta\succeq\mu$.

Therefore $\lambda\succeq\beta\succeq\mu$. Since $T$ contains $C(\mu)$, Lemma~\ref{lem:dominance} gives $C(\lambda)\subseteq T$. As $T$ was arbitrary, $C(\lambda)$ is $n$-unavoidable.
\end{proof}

\section{The upper bound}\label{sec:upper}

For a tournament $T$ and $x\in V(T)$, let $P_x=N_T^+(x)$ and $Q_x=N_T^-(x)$, and let $G_x$ be the bipartite graph with parts $P_x,Q_x$ in which $uv$ is an edge, for $u\in P_x$ and $v\in Q_x$, if and only if $u\to v$.

\begin{lemma}\label{lem:rootobstruction}
If $T$ contains a spanning claw $C(\lambda)$ rooted at $x$, then $\defi(G_x)\le\exc(\lambda)$.
\end{lemma}

\begin{proof}
Consider a branch $x=v_0\to v_1\to\cdots\to v_\ell$. Select every branch edge $v_jv_{j+1}$ with $v_j\in P_x$ and $v_{j+1}\in Q_x$. These edges form a matching: two selected edges on one branch could meet only if their common vertex belonged to both $P_x$ and $Q_x$, which is impossible, and distinct branches are internally disjoint. Since $v_1\in P_x$, if the branch contains a vertex of $Q_x$, then its first such vertex is matched. Hence at most $(\ell-2)^+$ vertices of $Q_x$ on this branch remain unmatched. Taking the union of these matchings over all branches leaves at most $\exc(\lambda)$ vertices of $Q_x$ unmatched, and the result follows.
\end{proof}

Let $H$ be a tournament on $[q]$ with positive integer weights $a_1,\ldots,a_q$, and write $a(X)=\sum_{i\in X}a_i$. For $i\in[q]$, let $O_i=N_H^+(i)$ and $I_i=N_H^-(i)$.

\begin{lemma}\label{lem:weightedtemplate}
Suppose that for each $i\in[q]$ there are sets $Y_i\subseteq I_i$ and $Z_i\subseteq O_i$ such that $Y_i\to Z_i$ and $a(Y_i)+a(Z_i)-a(O_i)\ge\rho$, where $\rho>0$. Let $m$ be a positive integer and let $e_1,\ldots,e_q$ be nonnegative integers with $E=\sum_i e_i$. Replace vertex $i$ by a tournament $V_i$ of order $a_i m+e_i$ and orient all arcs between classes according to $H$. Then every vertex $x$ of the resulting tournament satisfies $\defi(G_x)\ge\rho m-E$.
\end{lemma}

\begin{proof}
Fix $x\in V_i$, and let $\mathcal Y=\bigcup_{j\in Y_i}V_j$ and $\mathcal Z=\bigcup_{j\in Z_i}V_j$. We have $\mathcal Y\subseteq Q_x$, $\mathcal Z\subseteq P_x$, $\mathcal Y\to V_i$ and $\mathcal Y\to\mathcal Z$. Hence $\Gamma_{G_x}(\mathcal Y)\subseteq\bigcup_{j\in O_i\setminus Z_i}V_j$. Writing $e(X)=\sum_{j\in X}e_j$, we get
\begin{align*}
\defi(G_x)&\ge|\mathcal Y|-|\Gamma_{G_x}(\mathcal Y)|\\
&\ge\bigl(a(Y_i)+a(Z_i)-a(O_i)\bigr)m+e(Y_i)+e(Z_i)-e(O_i).
\end{align*}
Since $Z_i\subseteq O_i$, the last three terms equal $e(Y_i)-e(O_i\setminus Z_i)\ge-E$, which proves the claim.
\end{proof}

We use the tournament $H$ on $[12]$ specified in Table~\ref{tab:template}, with weights $(a_1,\ldots,a_{12})=(1,1,1,4,1,1,4,4,1,1,1,1)$, whose sum is $21$. The table also gives sets $Y_i\subseteq N_H^-(i)$ and $Z_i\subseteq O_i$. The template was located by computer search; no computation enters the proof, since Lemma~\ref{lem:weightedtemplate} uses only the exhibited witnesses $(Y_i,Z_i)$. Its cyclic symmetry of order $3$ makes the certificate particularly easy to verify.

\begin{table}[ht]
\centering
\scriptsize
\setlength{\tabcolsep}{2.5pt}
\renewcommand{\arraystretch}{1.12}
\begin{tabularx}{\textwidth}{c c >{\raggedright\arraybackslash}X >{\raggedright\arraybackslash}X >{\raggedright\arraybackslash}p{1.55cm} c}
\toprule
$i$ & $a_i$ & $O_i$ & $Y_i$ & $Z_i$ & $(a(O_i),a(Y_i),a(Z_i))$\\
\midrule
1  &1& $3,5,6,7,9,10,11$ & $8$ & $3,5,6,7$ & $(10,4,7)$\\
2  &1& $1,3,6,8,10,11,12$ & $4$ & $1,8,11,12$ & $(10,4,7)$\\
3  &1& $4,5,7,9,10$ & $1,6,8,11,12$ & $7$ & $(11,8,4)$\\
4  &4& $1,2,8,11,12$ & $3,5,6,7,9,10$ & $\varnothing$ & $(8,9,0)$\\
5  &1& $2,4,6,9,10,11,12$ & $7$ & $2,4,9,10$ & $(10,4,7)$\\
6  &1& $3,4,7,9,10$ & $1,8,11,12$ & $3,7$ & $(11,7,5)$\\
7  &4& $2,4,5,9,10$ & $1,3,6,8,11,12$ & $\varnothing$ & $(8,9,0)$\\
8  &4& $1,3,5,6,7$ & $2,4,9,10,11,12$ & $\varnothing$ & $(8,9,0)$\\
9  &1& $2,4,8,11,12$ & $3,5,6,7,10$ & $4$ & $(11,8,4)$\\
10 &1& $4,8,9,11,12$ & $3,5,6,7$ & $4,9$ & $(11,7,5)$\\
11 &1& $3,6,7,8,12$ & $2,4,9,10$ & $8,12$ & $(11,7,5)$\\
12 &1& $1,3,6,7,8$ & $2,4,9,10,11$ & $8$ & $(11,8,4)$\\
\bottomrule
\end{tabularx}
\caption{The weighted tournament template and its Hall witnesses. The rows form four orbits under the permutation $\pi$ used in Lemma~\ref{lem:certificate}.}
\label{tab:template}
\end{table}

\begin{lemma}\label{lem:certificate}
The data in Table~\ref{tab:template} define a tournament $H$. For every $i\in[12]$, we have $Y_i\subseteq N_H^-(i)$, $Z_i\subseteq O_i$, $Y_i\to Z_i$, and $a(Y_i)+a(Z_i)-a(O_i)=1$.
\end{lemma}

\begin{proof}
Let $\pi=(1\,2\,5)(4\,7\,8)(3\,12\,9)(6\,11\,10)$.
Inspection of the table gives $a_{\pi(i)}=a_i, O_{\pi(i)}=\pi(O_i), Y_{\pi(i)}=\pi(Y_i)$, and $Z_{\pi(i)}=\pi(Z_i)$
for every $i$. Since $\pi$ has order $3$ and no fixed point, every orbit of unordered pairs has size $3$ and contains a pair one of whose members lies in $\{1,3,4,6\}$; thus it suffices to check the four representatives $i\in\{1,3,4,6\}$. For these representatives, the $O_i$ column verifies that exactly one of $j\in O_i$ and $i\in O_j$ holds for every $j\ne i$, and also gives $Y_i\subseteq N_H^-(i)$ and $Z_i\subseteq O_i$. The witness relations are
\[
8\to\{3,5,6,7\},\quad
\{1,6,8,11,12\}\to7,\quad
\{3,5,6,7,9,10\}\to\varnothing,\quad
\{1,8,11,12\}\to\{3,7\},
\]
while the corresponding triples $(a(O_i),a(Y_i),a(Z_i))$ are
$(10,4,7),(11,8,4),(8,9,0)$, and $(11,7,5)$.
Each triple has its last two entries summing to one more than the first, and equivariance under $\pi$ proves all assertions.
\end{proof}

\begin{proof}[Proof of Theorem~\ref{thm:upper}]
Write $n=21m+\ell$, where $0\le\ell\le20$. Since $n\ge441$, we have $m\ge21\ge\ell+1$. Choose nonnegative integers $e_1,\ldots,e_{12}$ with $\sum_i e_i=\ell$, replace vertex $i$ of $H$ by a tournament $V_i$ of order $a_i m+e_i$, and orient all arcs between classes according to $H$. Denote the resulting tournament by $T_n$.

Lemmas~\ref{lem:weightedtemplate} and \ref{lem:certificate}, with $\rho=1$ and $E=\ell$, show that every $x\in V(T_n)$ satisfies $\defi(G_x)\ge m-\ell$. Set $r=m-\ell-1$ and $s=9m+2\ell+1$. Then $r,s\ge0$, $3r+2s=n-1$, and $r+s=10m+\ell$, so $C(3^r,2^s)$ is an $n$-vertex claw of the required degree. If it were contained in $T_n$ with root $x$, Lemma~\ref{lem:rootobstruction} would give $\defi(G_x)\le\exc(3^r,2^s)=r=m-\ell-1$, a contradiction.
\end{proof}

\begin{remark}
For $n=21m$, the construction avoids $C(3^{m-1},2^{9m+1})$, whose degree is exactly $10m=10n/21$.
\end{remark}

\section{Two bottleneck parameters}\label{sec:bottleneck}

For a tournament $T$ on $N+1$ vertices, let $\tau(T)$ be the minimum number of length-$3$ branches in a spanning claw whose branch lengths belong to $\{1,2,3\}$. This parameter is well-defined by Lemma~\ref{lem:maxroot}. Define
\[
\theta=\limsup_{N\to\infty}\frac1N\max_{|T|=N+1}\tau(T).
\]

\begin{proposition}\label{prop:theta}
We have
\[
\liminf_{n\to\infty}\frac{u(n)}n\ge\frac12-\frac{\theta}{2}.
\]
\end{proposition}

\begin{proof}
Fix $0<\varepsilon<1-\theta$. For all sufficiently large $N$, every tournament on $N+1$ vertices contains a spanning claw $C(\mu)$ with at most $(\theta+\varepsilon)N$ parts equal to $3$. If $C(\lambda)$ has $k\le\floor{(1-\theta-\varepsilon)N/2}$ branches and $\beta$ is the balanced $k$-part partition of $N$, then $C(\mu)$ has at least $k$ branches: if it has $b$ branches and $t$ branches of length $3$, then $N\le2b+t\le2b+(\theta+\varepsilon)N$. If $k\le N/3$, every part of $\beta$ is at least $3$; otherwise $\beta=(3^r,2^{k-r})$ with
\[
r=N-2k\ge(\theta+\varepsilon)N\ge t.
\]
In either case $\beta\succeq\mu$, and Lemmas~\ref{lem:balanced} and \ref{lem:dominance} show that every such $C(\lambda)$ is unavoidable. Hence
\[
u(N+1)\ge\floor{\frac{(1-\theta-\varepsilon)N}{2}}
\]
for all sufficiently large $N$. Dividing by $N+1$, taking the lower limit and then letting $\varepsilon\to 0$ proves the result.
\end{proof}

The upper-bound argument has a parallel formulation. Call the data in Lemma~\ref{lem:weightedtemplate} a weighted Hall certificate, and let $W=\sum_i a_i$ be its total weight and $\rho=\min_i\{a(Y_i)+a(Z_i)-a(O_i)\}$ its positive integer surplus.

\begin{proposition}\label{prop:ratio}
If a weighted Hall certificate of total weight $W$ and surplus $\rho>0$ exists, then
$\clawconstant\le\frac12-\frac{\rho}{2W}$.
\end{proposition}

\begin{proof}
Write $n=Wm+\ell$, where $0\le\ell<W$, and form a blow-up with class orders $a_i m+e_i$, where the nonnegative integers $e_i$ sum to $\ell$. Lemma~\ref{lem:weightedtemplate} gives $\defi(G_x)\ge D:=\rho m-\ell$ for every root $x$. For all sufficiently large $m$, choose $R\in\{D-1,D-2\}$ with $R\equiv n-1\pmod 2$, and let $k=(n-1-R)/2$. The partition $\lambda=(R+2,2^{k-1})$ has sum $n-1$ and excess $R<D$. Hence Lemma~\ref{lem:rootobstruction} shows that the blow-up avoids $C(\lambda)$, so $u(n)\le k-1$. Since
\[
\frac{k}{n}=\frac{(W-\rho)m+2\ell+O(1)}{2(Wm+\ell)}
=\frac12-\frac{\rho}{2W}+O\left(\frac1n\right),
\]
the assertion follows.
\end{proof}

Let $\sigma$ be the supremum of $\rho/W$ over all weighted Hall certificates. The ratio is unchanged if all weights are multiplied by a common positive integer, and the family of finite integer certificates is countable.

\begin{proposition}\label{prop:duality}
We have $\sigma\le\theta$. More precisely, a certificate of total weight $W$ and surplus $\rho$ gives, for every positive integer $m$, a tournament $T_m$ on $Wm$ vertices such that $\tau(T_m)\ge\rho m$.
\end{proposition}

\begin{proof}
Take the equal blow-up in Lemma~\ref{lem:weightedtemplate}, so $e_i=0$ for every $i$. Then $\defi(G_x)\ge\rho m$ for every root $x$. If a spanning claw in $T_m$ has all branch lengths in $\{1,2,3\}$ and has $t$ branches of length $3$, then its excess is exactly $t$, so Lemma~\ref{lem:rootobstruction} gives $t\ge\defi(G_x)\ge\rho m$. Hence $\tau(T_m)\ge\rho m$, and therefore
\[
\theta\ge\lim_{m\to\infty}\frac{\rho m}{Wm-1}=\frac{\rho}{W}.
\]
Taking the supremum over all certificates proves $\sigma\le\theta$.
\end{proof}

Combining Propositions~\ref{prop:theta}, \ref{prop:ratio} and \ref{prop:duality}, we obtain
\[
\frac12-\frac{\theta}{2}\le\liminf_{n\to\infty}\frac{u(n)}n
\le\clawconstant\le\frac12-\frac{\sigma}{2},
\qquad \sigma\le\theta.
\]
Theorem~\ref{thm:structural} and Table~\ref{tab:template} give $\frac1{21}\le\sigma\le\theta\le\frac15$.
Thus improving the lower bound amounts to decreasing $\theta$, while improving the upper bound amounts to increasing $\sigma$. If $\sigma=\theta$, then the limit of $u(n)/n$ exists and equals $1/2-\theta/2$. It is therefore natural to ask whether $\theta=1/5$, that is, whether Theorem~\ref{thm:structural} is asymptotically sharp, and whether $\sigma=\theta$.

In the parametrisation $1/2-\alpha/2$, the previous upper bound $11/23$ corresponds numerically to $\alpha=1/23$, whereas Table~\ref{tab:template} gives $\rho/W=1/21$. For fixed $q$ and $W$, maximizing $\rho$ is a finite optimization problem over $q$-vertex tournaments, positive integer weights summing to $W$, and witness sets $Y_i,Z_i$. Both proofs are constructive, and the upper-bound certificate can be verified directly from Table~\ref{tab:template}.


\begin{thebibliography}{99}
\small\setlength{\itemsep}{0pt}

\bibitem{Lu1991}
X. Lu, \emph{On claws belonging to every tournament}, Combinatorica \textbf{11} (1991), 173--179.

\bibitem{Lu1993}
X. Lu, \emph{Claws contained in all $n$-tournaments}, Discrete Math. \textbf{119} (1993), 107--111.

\bibitem{Lu1996}
X. Lu, \emph{On avoidable and unavoidable trees}, J. Graph Theory \textbf{22} (1996), 335--346.

\bibitem{LuWangWong}
X. Lu, D.-W. Wang and C.-K. Wong, \emph{On avoidable and unavoidable claws}, Discrete Math. \textbf{184} (1998), 259--265.

\bibitem{MycroftNaia}
R. Mycroft and T. Naia, \emph{Unavoidable trees in tournaments}, Random Structures Algorithms \textbf{53} (2018), 352--385.

\bibitem{Ore1955}
O. Ore, \emph{Graphs and matching theorems}, Duke Math. J. \textbf{22} (1955), 625--639.

\bibitem{Redei}
L. R\'edei, \emph{Ein kombinatorischer Satz}, Acta Litt. Sci. Szeged \textbf{7} (1934), 39--43.

\bibitem{SaksSos}
M. E. Saks and V. T. S\'os, \emph{On unavoidable subgraphs of tournaments}, in \emph{Finite and Infinite Sets, Vol. I, II (Eger, 1981)}, Colloq. Math. Soc. J\'anos Bolyai, vol. 37, North-Holland, Amsterdam, 1984, 663--674.

\bibitem{SteinSurvey}
M. Stein, \emph{Oriented trees and paths in digraphs}, in \emph{Surveys in Combinatorics 2024}, London Math. Soc. Lecture Note Ser., vol. 493, Cambridge Univ. Press, Cambridge, 2024, 271--295.

\end{thebibliography}
\end{document}